\documentclass{amsart}

\newtheorem{theorem}{Theorem}
\theoremstyle{plain}

\newtheorem{corollary}{Corollary}

\newtheorem{lemma}{Lemma}

\numberwithin{equation}{section}

\begin{document}

\noindent
\title[Gr\"{u}ss type inequality]{A Gr\"{u}ss type inequality for vector-value functions in Hilbert $C^*$-modules}
\newline
\author[A. G. Ghazanfari]{Amir Ghasem Ghazanfari}

\address{Department of Mathematics, Lorestan University, P. O. Box 465, Khoramabad, Iran.}

\email{ghazanfari.amir@gmail.com}

\subjclass[2000]{46L08, 46H25, 26D15.}

\keywords{Hilbert $C^*$-modules, Gr\"{u}ss inequality, Landau type inequality, Bochner integral.}

\begin{abstract}
In this paper we prove a version of Gr\"{u}ss' integral inequality
for mappings with values in Hilbert $C^*$-modules. Some applications for such functions are also given.
\end{abstract}

\maketitle

\section{introduction}
\noindent
In 1934, G. Gr\"{u}ss \cite{gru} showed that for two Lebesgue integrable
functions $f,g:[a,b]\rightarrow \mathbb{R}$,
\begin{equation*}
\left\vert \frac{1}{b-a}\int_{a}^{b}f(t)g(t)dt-\frac{1}{b-a}%
\int_{a}^{b}f(t)dt\frac{1}{b-a}\int_{a}^{b}g(t)dt\right\vert \leq \frac{1}{4}%
(M-m)(N-n),
\end{equation*}%
provided $m,M,n,N$ are real numbers with the property $-\infty <m\leq f\leq
M<\infty $ and $-\infty <n\leq g\leq N<\infty \quad \text{a.e. on }[a,b].$
The constant $\frac{1}{4}$ is best possible in the sense that it cannot be
replaced by a smaller constant.

The following inequality of Gr\"{u}ss type in real or complex inner product
spaces is known \cite{dra1}.

\begin{theorem}\label{t1}
Let $(H;\left\langle \cdot ,\cdot \right\rangle )$ be an inner
product space over $\mathbb{K}\quad (\mathbb{K}=\mathbb{C},\mathbb{R})$ and $%
e\in H,\Vert e\Vert =1$. If $\alpha ,\beta ,\lambda ,\mu \in \mathbb{K}$ and
$x,y\in H$ are such that conditions
\begin{equation*}
Re\langle \alpha e-x,x-\beta e\rangle \geq 0,\quad Re\langle
\lambda e-y,y-\mu e\rangle \geq 0
\end{equation*}%
or, equivalently,
\begin{equation*}
\left\Vert x-\frac{\alpha +\beta }{2}e\right\Vert \leq \frac{1}{2}|\alpha
-\beta |,\quad \left\Vert y-\frac{\lambda +\mu }{2}e\right\Vert \leq \frac{1%
}{2}|\lambda -\mu |
\end{equation*}%
hold, then the following inequality holds
\begin{equation}\label{1.1}
\left\vert \left\langle x,y\right\rangle -\left\langle x,e\right\rangle
\left\langle e,y\right\rangle \right\vert \leq \frac{1}{4}|\alpha -\beta
||\lambda -\mu |.
\end{equation}%
The constant $\frac{1}{4}$ is best possible in (\ref{1.1}).
\end{theorem}
Let $\langle H;\langle .,.\rangle\rangle$ be a real or complex Hilbert space, $\Omega\subset \mathbb{R}^{n}$ be a Lebesgue measurable set and $\rho:\Omega\longrightarrow [0,\infty)$ a Lebesgue measurable function with $\int_{\Omega}\rho(s)ds=1$. We denote by $L_{2,\rho}(\Omega ,H)$ the set of all strongly measurable functions $f$ on $\Omega$ such that $\| f\|_{2,\rho}^{2}:=\int_{\Omega}\rho(s)\| f(s)\|^{2}ds<\infty$.

A further extension of Gr\"{u}ss type inequality for Bochner integrals of vector-valued functions
in real or complex Hilbert spaces is given in \cite{bus}.

\begin{theorem}\label{t2}
Let $\langle H;\langle .,.\rangle\rangle$ be a real or complex Hilbert space, $\Omega\subset \mathbb{R}^{n}$ be a Lebesgue measurable set and $\rho:\Omega\longrightarrow [0,\infty)$ a Lebesgue measurable function with $\int_{\Omega}\rho(s)ds=1$. If $f,g$ belong to $L_{2,\rho}(\Omega,H)$ and there exist the vectors $x, X, y, Y\in H$ such that
\begin{align}\label{1.2}
&\int_{\Omega}\rho(t) Re\langle X-f(t),f(t)-x\rangle dt\geq 0,\\
&\int_{\Omega}\rho(t)Re\langle Y-g(t),g(t)-y\rangle dt\geq 0,\notag
\end{align}
or, equivalently,
\begin{align}\label{1.3}
\int_{\Omega}\rho(t)\left\Vert f(t)-\dfrac{X+x}{2}\right\Vert^{2}dt&\leq \dfrac{1}{4}\Vert X-x\Vert^{2},\\
\int_{\Omega}\rho(t)\left\Vert g(t)-\dfrac{Y+y}{2}\right\Vert^{2}dt&\leq \dfrac{1}{4}\Vert Y-y\Vert^{2}.\notag
\end{align}
Then the following inequalities hold
\begin{align}\label{1.4}
&\left\vert\int_{\Omega}\rho(t)\langle f(t),g(t)\rangle dt -\left\langle\int_{\Omega}\rho(t)f(t)dt,\int_{\Omega}\rho(t)g(t)dt\right\rangle\right\vert\notag\\
&\leq \dfrac{1}{4}\Vert X-x\Vert\Vert Y-y\Vert-\left[\int_{\Omega}\rho(t)Re\langle X-f(t),f(t)-x\rangle dt\right.\\
&\left.\times\int_{\Omega}\rho(t)Re\langle Y-g(t),g(t)-y\rangle dt\right]^{\frac{1}{2}}\notag\\
&\leq \dfrac{1}{4}\Vert X-x\Vert\Vert Y-y\Vert.\notag
\end{align}
The constant $\frac{1}{4}$
is sharp in the sense mentioned above.
\end{theorem}

The Gr\"uss inequality has been investigated in inner product modules over $H^*$-algebras and $C^*$-algebras
\cite{ban,gha1}, completely bounded maps \cite{pec}, $n$-positive linear maps \cite{mos} and
semi-inner product $C\sp *$-modules \cite{fuj}.

Also Joci\'c et.al. in \cite{joc} presented the following Gr\"{u}ss type inequality
\begin{equation*}
\left|\left|\left|\int_\Omega \mathcal{A}_tX\mathcal{B}_td\mu(t)-\int_\Omega \mathcal{A}_td\mu(t)X\int_\Omega \mathcal{B}_td\mu(t)\right|\right|\right|\leq \frac{\|D-C\|.\|F-E\|}{4}|||X|||
\end{equation*}
for all bounded self-adjoint fields satisfying $C\leq \mathcal{A}_t \leq D$ and $E \leq \mathcal{B}_t \leq F$ for all $t\in \Omega$ and some bounded self-adjoint operators $C,D,E$ and $F$, and for all $X\in\mathcal{C}_{|||.|||}(H)$.

The main aim of this paper is to obtain a generalization of Theorem \ref{t2} for vector-value functions in Hilbert $C^*$-modules.
Some applications for such functions are also given.
\section{Preliminaries}

Hilbert $C^*$-modules are used as the framework for Kasparov's bivariant K-theory and form the
technical underpinning for the $C^*$-algebraic approach to quantum groups. Hilbert $C^*$-modules
are very useful in the following research areas: operator K-theory, index theory for operator valued
conditional expectations, group representation theory, the theory of $AW^*$-algebras, noncommutative
geometry, and others. Hilbert $C^*$-modules form a category in between Banach
spaces and Hilbert spaces and obey the same axioms as a Hilbert space except that the inner
product takes values in a general $C^*$-algebra than the complex number $\mathbb{C}$. This simple generalization
gives a lot of trouble. Fundamental and familiar Hilbert space properties like Pythagoras' equality, self-duality and
decomposition into orthogonal complements must be given up. Moreover,
a bounded module map between Hilbert $C^*$-modules need not have an adjoint; not every
adjointable operator need have a polar decomposition. Hence to get its applications, we have to
use it with great care.

Let $\mathcal{A}$ be a $C^*$-algebra. A semi-inner product module
over $\mathcal{A}$ is a right module $X$ over $\mathcal{A}$ together with a generalized semi-inner product,
that is with a mapping $\langle.,.\rangle$ on $X\times X$, which is $\mathcal{A}$-valued and having the following properties:
\begin{enumerate}
\item[(i)]
 $\langle x,y+z\rangle=\langle x,y\rangle+\langle x,z\rangle$ for all $x,y,z\in X,$
\item[(ii)]
$\left\langle x,ya\right\rangle=\left\langle x,y\right\rangle a$ for $x,y\in X, a\in \mathcal{A}$,
\item[(iii)]
$\langle x,y\rangle ^*=\langle y,x\rangle$ for all $x,y\in X$,
\item[(iv)]
$\left\langle x,x\right\rangle \geq 0$ for $x\in X$.
\end{enumerate}
We will say that $X$ is a semi-inner product $C^*$-module.
The absolute value of $x\in X$ is defined as the square root of $\langle x,x\rangle$ and it is denoted
by $|x|$.
If, in addition,
\begin{enumerate}
\item[(v)]
$\langle x,x\rangle=0$ implies $x=0$,
\end{enumerate}
then $\langle .,.\rangle$ is called a generalized inner
product and $X$ is called an inner product module over $\mathcal{A}$ or an inner product $C^*$-module.
An Inner product $C^*$-module which is complete with respect to the norm $\|x\|:=\|\langle x,x\rangle\|^\frac{1}{2}\quad (x\in X)$
is called a Hilbert $C^*$-module.

As we can see, an inner product module obeys the same axioms as an ordinary
inner product space, except that the inner product takes values in a more general
structure than in the field of complex numbers.

If $\mathcal{A}$ is a $C^*$-algebra and $X$ is a semi-inner product $\mathcal{A}$-module, then the following Schwarz
inequality holds:
\begin{equation}\label{2.1}
\langle x,y\rangle\langle y,x\rangle\leq \|\langle x,x\rangle\|~\langle y,y\rangle~~(x,y\in X).
\end {equation}
(e.g. \cite[Proposition 1.1]{lan}).

It follows from the Schwarz inequality (\ref{2.1}) that $\|x\|$ is a semi-norm on $X$.

Now let $\mathcal{A}$ be a $\ast $-algebra, $\varphi $ a positive linear functional on $%
\mathcal{A}$ and $X$ be a semi-inner $\mathcal{A}$-module. We can define a
sesquilinear form on $X\times X$ by $\sigma (x,y)=\varphi \left(
\left\langle x,y\right\rangle \right) $; the Schwarz inequality for $\sigma $
implies that
\begin{equation}
\vert \varphi\langle x,y\rangle\vert^{2}\leq \varphi\langle x,x\rangle \varphi\langle y,y\rangle.  \label{2.4}
\end{equation}%
In \cite[Proposition 1, Remark 1]{gha} the authors present two other forms of the Schwarz inequality in
semi-inner $\mathcal{A}$-module $X$, one for
positive linear functional $\varphi$ on $\mathcal{A}$:
\begin{equation}
\varphi(\langle x,y\rangle\langle y,x\rangle)
\leq \varphi \langle x,x\rangle r\langle y,y\rangle,  \label{2.5}
\end{equation}%
where $r$ is spectral radius, another one for $C^*$-seminorm $\gamma$ on $\mathcal{A}$:
\begin{equation}
(\gamma\langle x,y\rangle) ^{2}\leq \gamma\langle x,x\rangle \gamma \langle y,y\rangle.  \label{2.6}
\end{equation}%


\section{The main results}
Let $\mathcal{A}$ be a $C^*$-algebra, first we state some basic properties of integrals of $\mathcal{A}$-value functions with respect to a positive measure for Bochner integrability of functions which we need to use them in our discussion. For basic properties of integrals of vector value functions with respect to scalar measures and integrals of scalar value functions with respect to vector measures see chapter II in \cite{die}.
 \begin{lemma}
Let $\mathcal{A}$ be a $C^*$-algebra, $X$ a Hilbert $C^*$-module and $(\Omega,\mathfrak{M},\mu)$ be a measure space.
If $f: \Omega\rightarrow X$ is Bochner integrable and $a\in \mathcal{A}$ then
\begin{enumerate}
\item[(a)] $fa$ is Bochner integrable, where $fa(t)=f(t)a~~(t\in\Omega)$,
\item[(b)] the function $f^*:\Omega \rightarrow \mathcal{A}$ defined by $f^*(t)=(f(t))^*$ is Bochner integrable and $\int_\Omega f^*d\mu=\left(\int_\Omega f~d\mu\right)^*$.\\
    Furtheremore,
\item[(c)] If $\mu(\Omega)<\infty$ and $f$ is positive, i.e., $f(t)\geq 0$ for all $t\in \Omega$, then\\ $\int_\Omega f(t) d\mu(t)\geq0.$
\end{enumerate}
\end{lemma}
\begin{proof} (a) Suppose $\varphi: \Omega\rightarrow X$ is a simple function with finite support i.e., $ \varphi=\sum_{i=1}^n \chi_{E_i} x_i $ with $\mu(E_i)<\infty$ for each non-zero $x_i\in X$ $(i=1,2,3,...,n)$, then for every $a\in \mathcal{A}$ the function $\varphi a :\Omega\rightarrow X$ defined by $\varphi a(t)=\varphi(t)a~~~(t\in\Omega)$ is a simple function and $\int_\Omega \varphi a~ d\mu=\left(\int_\Omega \varphi d\mu\right)a$.
 Since $f: \Omega\rightarrow X$ is Bochner integrable and  $\|f(t)a\|\leq\|f(t)\|~\|a\|$ thus $fa$ is strongly measurable and $\int_\Omega fa~ d\mu=\left(\int_\Omega f d\mu\right)a$, therefore $fa$ is Bochner integrable.

\item[(b)] For every simple function
$ \varphi=\sum_{i=1}^n \chi_{E_i} a_i $ we have $ \varphi^*=\sum_{i=1}^n \chi_{E_i} a_i^* $ consequently $\int_\Omega \varphi^*d\mu=\left(\int_\Omega \varphi~d\mu\right)^*$. The result therefore follows.

\item[(c)] Suppose that $\mu(\Omega)<\infty$ and $f$ is a Bochner integrable function and positive, i.e., $f(t)\geq 0$ for all $t\in \Omega$.
Since $f$ is Bochner integrable $\int_\Omega \|f(t)\| d\mu(t)<\infty$ by Theorem 2 in \cite[chapter II, section 2]{die}. Using Holder inequality for Lebesgue integrable functions we get
\[
\int_\Omega \|f^{\frac{1}{2}}(t)\| d\mu(t)=\int_\Omega \|f(t)\|^\frac{1}{2} d\mu(t)\leq \left(\int_\Omega \|f(t)\| d\mu(t)\right)^\frac{1}{2}\mu(\Omega)^\frac{1}{2}<\infty.
\]
So $f^\frac{1}{2}$ is Bochner integrable thus there is a sequence of simple functions $\varphi_n$ such that $\varphi_n(t)\rightarrow f^\frac{1}{2}(t)$ for almost all $t\in\Omega$ and
$\int_\Omega\varphi_n(t)d\mu(t)\rightarrow \int_\Omega f^\frac{1}{2}(t)d\mu(t)$ in norm topology in $\mathcal{A}$.

This implies that $\varphi_n(t)^*\varphi_n(t)\rightarrow f(t)$, i.e., for every positive Bochner integrable function $f$ there is a sequence of positive simple functions $\psi_n$ such that $\psi_n(t)\rightarrow f(t)$ for almost all $t\in\Omega$ and
$\int_\Omega\psi_n(t)d\mu(t)\rightarrow \int_\Omega f(t)d\mu(t)$ in norm topology in $\mathcal{A}$.
By proposition (1.6.1) in \cite{dix} the set of positive elements in a $C^*$-algebra is a closed convex cone, therefore $\int_\Omega f(t) d\mu(t)\geq0$ since $\int_\Omega \psi_n(t) d\mu(t)\geq0$.
\end{proof}

If $\mu$ is a probability measure on $\Omega$. We denote by $L_{2}(\Omega ,X)$ the set of all strongly measurable functions $f$ on $\Omega$ such that $\| f\|_{2}^{2}:=\int_{\Omega}\| f(s)\|^{2}d\mu(s)<\infty$.

For every $a\in X$, we define the constant function $e_a\in L_{2}(\Omega,X)$ by $e_a(t)=a ~~(t\in \Omega)$.
In the following lemma we show that a special kind of invariant property holds, which we will use in the sequel.

\begin{lemma}\label{l1}
If $f,g\in L_{2}(\Omega,X)$, $a,b\in X$ and $e_a,e_b$ are measurable then
\begin{multline}\label{3.1}
\int_{\Omega}\langle f(t)-e_a(t),g(t)-e_b(t)\rangle d\mu(t)-\left\langle\int_{\Omega}(f(t)-e_a(t))d\mu(t),\int_{\Omega}(g(t)-e_b(t))d\mu(t)\right\rangle\\
=\int_{\Omega}\langle f(t),g(t)\rangle d\mu(t)-\left\langle\int_{\Omega}f(t)d\mu(t),\int_{\Omega}g(t)d\mu(t)\right\rangle.
\end{multline}
In particular,
\begin{equation}\label{3.2}
\int_{\Omega}| f(t)|^{2}d\mu(t)-\left|\int_{\Omega}f(t)d\mu(t)\right|^{2}\leq\int_{\Omega}| f(t)-e_a(t)|^2 d\mu(t).
\end{equation}
\end{lemma}

\begin{proof}
We must state that the functions under the integrals
(\ref{3.1}) are Bochner integrable on $\Omega$ since they are strongly
measurable and we can state the following obvious results:

For every $\Lambda\in X^*$ we have $\Lambda\left(\int_{\Omega}f(t)d\mu(t)\right)=\int_{\Omega}\Lambda(f(t))d\mu(t)$.
Therefore
\begin{align*}
\int_{\Omega}\langle f(t),b\rangle d\mu(t)&=\left\langle\int_{\Omega}f(t)d\mu(t),b\right\rangle,\\
\int_{\Omega}\langle a,g(t)\rangle d\mu(t)&=\left\langle a, \int_{\Omega}g(t)d\mu(t)\right\rangle.
\end{align*}

Also for almost all $t\in\Omega$ we have
\begin{equation*}
\int_{\Omega}\Vert f(t)\Vert d\mu(t)
\leq\Big(\mu(\Omega)\Big)^\frac{1}{2}\left(\int_{\Omega}\Vert f(t)\Vert^2d\mu(t)\right)^{\frac{1}{2}}=\| f\|_{2},
\end{equation*}
\begin{equation*}
\int_{\Omega}\Vert g(t)\Vert d\mu(t)
\leq\Big(\mu(\Omega)\Big)^\frac{1}{2}\left(\int_{\Omega}\Vert g(t)\Vert^2d\mu(t)\right)^{\frac{1}{2}}=\| g\|_{2},
\end{equation*}
and
\begin{equation*}
\int_{\Omega}\Vert \langle f(t),g(t)\rangle\Vert d\mu(t)\leq\| f\|_{2}\|g\|_2.
\end{equation*}

A simple calculation shows that
\begin{multline*}
\int_{\Omega}\langle f(t)-e_a(t),g(t)-e_b(t)\rangle d\mu(t)-\left\langle\int_{\Omega}(f(t)-e_a(t))d\mu(t),\int_{\Omega}(g(t)-e_b(t))d\mu(t)\right\rangle\\
=\int_{\Omega}\Big(\langle f(t), g(t)\rangle -\langle f(t), b\rangle-\langle a, g(t)\rangle+\langle a,b\rangle\Big)d\mu(t)\\
-\left\langle\int_{\Omega}f(t)d\mu(t)-a,\int_{\Omega}g(t)d\mu(t)-b\right\rangle\\
=\int_{\Omega}\langle f(t),g(t)\rangle d\mu(t) -\left\langle\int_{\Omega}f(t)d\mu(t),\int_{\Omega}g(t)d\mu(t)\right\rangle,
\end{multline*}
and for $f=g$ and $a=b$ we deduce (\ref{3.2}).
\end{proof}

The following result concerning a generalized semi-inner product on $L_{2}(\Omega,X)$ may be stated:
\begin{lemma}\label{l2}
If $f,g\in L_{2}(\Omega,X)$, $x, x', y, y'\in X$, then
\begin{enumerate}
\item[(i)] the following inequalities (\ref{3.3}) and (\ref{3.4}) are equivalent
\begin{align}\label{3.3}
&\int_{\Omega} Re\langle x'-f(t),f(t)-x\rangle d\mu(t)\geq 0,\\
&\int_{\Omega}Re\langle y'-g(t),g(t)-y\rangle d\mu(t)\geq 0,\notag
\end{align}

\begin{align}\label{3.4}
\int_{\Omega}\left| f(t)-\dfrac{x'+x}{2}\right|^{2}d\mu(t)&\leq \dfrac{1}{4}| x'-x|^{2},\\
\int_{\Omega}\left| g(t)-\dfrac{y'+y}{2}\right|^{2}d\mu(t)&\leq \dfrac{1}{4}| y'-y|^{2}.\notag
\end{align}
\item[(ii)] The map $[f,g]: L_{2}(\Omega,X)\times L_{2}(\Omega,X)\rightarrow \mathcal{A},$
\begin{equation}\label{3.5}
[f,g]:=\int_{\Omega}\langle f(t),g(t)\rangle d\mu(t) -\left\langle\int_{\Omega}f(t)d\mu(t),\int_{\Omega}g(t)d\mu(t)\right\rangle,
\end{equation}
is a generalized semi-inner product on $L_{2}(\Omega,X)$.
\end{enumerate}
\end{lemma}

\begin{proof}
\item[(i)]
If $ f\in L_{2}(\Omega,X)$, since for any $y, x, x'\in X$
\begin{align*}
\left| y-\dfrac{x'+x}{2}\right|^{2}-\dfrac{1}{4}|x'-x|^{2}=Re\langle y-x', y-x\rangle,
\end{align*}
hence
\begin{align}\label{3.6}
&\int_{\Omega}Re\langle x'-f(t), f(t)-x\rangle d\mu(t)\notag\\
&=\int_{\Omega}\left[ \dfrac{1}{4}| x'-x|^{2}-\left|f(t)-\dfrac{x'+x}{2}\right|^{2}\right] d\mu(t)\\
&=\dfrac{1}{4}| x'-x|^{2}-\int_{\Omega}\left| f(t)-\dfrac{x'+x}{2}\right|^{2}d\mu(t)\notag
\end{align}
showing that, indeed, the inequalities (\ref{3.3}) and (\ref{3.4}) are equivalent.\\
\item[(ii)]
We note that the first integral in (\ref{3.5}) is belong to $\mathcal{A}$ and later integrals are in $X$ and
the following Korkine type identity for Bochner integrals holds:
\begin{align}\label{3.7}
&\int_{\Omega}\langle f(t),g(t)\rangle d\mu(t) -\left\langle\int_{\Omega}f(t)d\mu(t),\int_{\Omega}g(t)d\mu(t)\right\rangle\\
&=\frac{1}{2}\int_{\Omega}\int_{\Omega}\langle f(t)-f(s), g(t)-g(s)\rangle d\mu(t)d\mu(s).\notag
\end{align}
By an application of the identity (\ref{3.7}),
\begin{equation}\label{3.8}
\int_{\Omega}|f(t)|^2 d\mu(t) -\left|\int_{\Omega}f(t)d\mu(t)\right|^2
=\frac{1}{2}\int_{\Omega}\int_{\Omega}| f(t)-f(s)|^2 d\mu(t)d\mu(s)\geq 0.
\end{equation}

It is easy to show that $[ .,.]$ is a generalized semi-inner product on $L_{2}(\Omega,X)$.
\end{proof}

The following theorem is a generalization of Theorem \ref{t2} for Hilbert $C^*$-modules.

\begin{theorem}\label{t3}
Let $X$ be a Hilbert $C^*$-module, $\mu$ a probability measure on $\Omega$. If $f,g$ belong to $L_{2}(\Omega,X)$ and there exist the vectors $x, x', y, y'\in X$ such that
\begin{align}\label{3.9}
&\int_{\Omega} Re\langle x'-f(t),f(t)-x\rangle d\mu(t)\geq 0,\\
&\int_{\Omega}Re\langle y'-g(t),g(t)-y\rangle d\mu(t)\geq 0,\notag
\end{align}
or, equivalently,
\begin{align}\label{3.10}
\int_{\Omega}\left| f(t)-\dfrac{x'+x}{2}\right|^{2}d\mu(t)&\leq \dfrac{1}{4}| x'-x|^{2},\\
\int_{\Omega}\left| g(t)-\dfrac{y'+y}{2}\right|^{2}d\mu(t)&\leq \dfrac{1}{4}| y'-y|^{2}.\notag
\end{align}
Then the following inequalities hold
\begin{multline}\label{3.11}
\left\|\int_{\Omega}\langle f(t),g(t)\rangle d\mu(t) -\left\langle\int_{\Omega}f(t)d\mu(t),\int_{\Omega}g(t)d\mu(t)\right\rangle\right\|\\
\leq\left\|\int_{\Omega}| f(t)|^{2}d\mu(t)-\left|\int_{\Omega}f(t)d\mu(t)\right|^{2}\right\|^\frac{1}{2}
\left\|\int_{\Omega}| g(t)|^{2}d\mu(t)-\left|\int_{\Omega}g(t)d\mu(t)\right|^{2}\right\|^\frac{1}{2}\\
\leq \left\|\dfrac{1}{4}| x'-x|^2-\int_{\Omega}Re\langle x'-f(t),f(t)-x\rangle d\mu(t)\right\|^{\frac{1}{2}}\\
\times\left\|\dfrac{1}{4}| y'-y|^2-\int_{\Omega}Re\langle y'-g(t),g(t)-y\rangle d\mu(t)\right\|^{\frac{1}{2}}\\
\leq \dfrac{1}{4}\Vert x'-x\Vert\Vert y'-y\Vert.
\end{multline}
The coefficient $1$ in second inequality and constant $\frac{1}{4} $ in the last inequality are sharp in the sense that they cannot be replaced by a smaller quantity.
\end{theorem}

\begin{proof}
Since (\ref{3.5}) is a generalized semi-inner product on $L_{2}(\Omega,X)$, so Schwarz inequality holds, i.e.,
\begin{equation}\label{3.12}
\big\|[ f,g]\big\|^2\leq\big\| [ f,f]\big\|\big\|[ g,g]\big\|.
\end{equation}

Using (\ref{3.2}) with $a=\frac{x+x'}{2}$ and (\ref{3.6}) we get
\begin{align}\label{3.13}
[ f,f]&=\int_{\Omega}| f(t)|^{2}d\mu(t)-\left|\int_{\Omega}f(t)d\mu(t)\right|^{2}\notag\\
&\leq \int_{\Omega}\left| f(t)-\dfrac{x'+x}{2}\right|^{2}d\mu(t)\\
&=\dfrac{1}{4}| x'-x|^{2}-\int_{\Omega}Re\langle x'-f(t), f(t)-x\rangle d\mu(t)\notag\\
&\leq\dfrac{1}{4}| x'-x|^{2}.\notag
\end{align}
Similarly,
\begin{align}\label{3.14}
[ g,g]&=\int_{\Omega}| g(t)|^{2}d\mu(t)-\left|\int_{\Omega}g(t)d\mu(t)\right|^{2}\notag\\
&\leq \int_{\Omega}\left| g(t)-\dfrac{y'+y}{2}\right|^{2}d\mu(t)\\
&=\dfrac{1}{4}| y'-y|^{2}-\int_{\Omega}Re\langle y'-g(t), g(t)-y\rangle d\mu(t)\notag\\
&\leq\dfrac{1}{4}| y'-y|^{2}.\notag
\end{align}

By Schwarz inequality (\ref{3.12}) and inequalities (\ref{3.13}), (\ref{3.14}) we deduce (\ref{3.11}).

Now, suppose that (\ref{3.11}) holds with the constants $C,D>0$ in the third and forth inequalities. That is,
\begin{multline}\label{3.15}
\left\|\int_{\Omega}\langle f(t),g(t)\rangle d\mu(t) -\left\langle\int_{\Omega}f(t)d\mu(t),\int_{\Omega}g(t)d\mu(t)\right\rangle\right\|\\
\leq C \left\|\dfrac{1}{4}| x'-x|^2-\int_{\Omega}Re\langle x'-f(t),f(t)-x\rangle d\mu(t)\right\|^{\frac{1}{2}}\\
\times\left\|\dfrac{1}{4}| y'-y|^2-\int_{\Omega}Re\langle y'-g(t),g(t)-y\rangle d\mu(t)\right\|^{\frac{1}{2}}\\
\leq D\Vert x'-x\Vert\Vert y'-y\Vert.
\end{multline}

Every Hilbert space $H$ can be regarded as a Hilbert $\mathbb{C}$-module. If we choose $\Omega=[0,1]\subseteq \mathbb{R}, X=\mathbb{C}, f,g: [0,1]\rightarrow \mathbb{R}\subseteq X$,
\begin{equation}\label{3.16}
f(t)=g(t) =
\begin{cases}
-1 & \text{ if \( 0\leq t\leq \frac{1}{2} \)}\\
~1 & \text{if \(\frac{1}{2}\leq t\leq 1 \)}\\
\end{cases}
\end{equation}
then for $x'=y'=1, x=y=-1$ and $\mu$ Lebesgue measure on $\Omega$ the conditions (\ref{3.10}) holds. By (\ref{3.15})
we deduce
\[
1\leq C\leq 4D
\]
giving $C\geq 1$ and $D\geq \frac{1}{4}$, and the theorem is proved.

\end{proof}

\section{Applications}

\textbf{1.} Let $X$ be a Hilbert $C^*$-module and $\mathcal{B}(X)$ the set of all adjointable operators on $X$.
We recall that if $A \in \mathcal{B}(X)$ then its operator norm is
defined by
\[
\|A\|=\sup\{\|Ax\|: x\in X, \|x\|\leq 1\},
\]
with this norm $\mathcal{B}(X)$ is a $C^*$-algebra.

Let $\Omega=[0,1]$ and $f(t)=e^{tA}$ for $t\in \Omega$, where $A$ is an
invertible element in $\mathcal{B}(X)$. Since for each $t\in[0,1]$ one has
\[
\|e^{tA}\|\leq e^{t\|A\|}\leq e^{\|A\|},
\]
then an application of first inequality in (\ref{3.13}) for $x'=2e^{A},~x=-e^{A}$ gives:
\begin{equation*}
0\leq \int_0^1 \left|e^{tA}\right|^2dt-\left|\int_0^1e^{tA}dt\right|^2\leq \frac{9}{4}\left|e^{A}\right|^2.
\end{equation*}
This implies that
\begin{equation*}
\int_0^1 \left|e^{tA}\right|^2dt\leq \frac{9}{4}\left|e^{A}\right|^2+\left|A^{-1}(e^A-I)\right|^2.
\end{equation*}

\textbf{2.} For square integrable functions $f$ and $g$ on $[0,1]$ and $$D( f ,g) =\int_0^1f (t)g(t)dt-\int_0^1f(t)dt\int_0^1g(t)dt$$
Landau proved (see \cite{land})
\[
|D( f ,g)|\leq \sqrt{D( f , f )}\sqrt{D(g,g)}.
\]

Joci\'c et.al. in \cite{joc} have proved for a probability measure $\mu$ and for square integrable fields $(\mathcal{A}_t)$ and $(\mathcal{B}_t)$ $(t\in \Omega)$ of commuting normal operators the following Landau type inequality holds
\begin{multline*}
\left|\left|\left|\int_\Omega \mathcal{A}_tX\mathcal{B}_td\mu(t)-\int_\Omega \mathcal{A}_td\mu(t)X\int_\Omega \mathcal{B}_td\mu(t)\right|\right|\right|\\
\leq\left|\left|\left|\sqrt{\int_\Omega |\mathcal{A}_t|^2d\mu(t)-\left|\int_\Omega \mathcal{A}_td\mu(t)\right|^2}X
\sqrt{\int_\Omega |\mathcal{B}_t|^2d\mu(t)-\left|\int_\Omega \mathcal{B}_td\mu(t)\right|^2}\right|\right|\right|
\end{multline*}
for all $X\in B(H)$ and for all unitarily invariant norms $|||.|||$.

Every $C^*$-algebra can be regarded as a Hilbert $C^*$-modules over itself with the inner product defined by $\langle a,b\rangle=a^*b$.
 If we apply the first inequality in (\ref{3.11}) of Theorem \ref{t3}, we obtain the following result.

\begin{corollary}\label{c2}
Let $\mathcal{A}$ be a $C^*$-algebra, $\mu$ a probability measure on $\Omega$. If $f,g$ belong to $L_{2}(\Omega,\mathcal{A})$
Then the following inequality holds
\begin{multline}\label{6.1}
\left\|\int_{\Omega} f(t)g(t) d\mu(t) -\int_{\Omega}f(t)d\mu(t)\int_{\Omega}g(t)d\mu(t)\right\|\\
\leq\left\|\int_{\Omega}| f(t)|^{2}d\mu(t)-\left|\int_{\Omega}f(t)d\mu(t)\right|^{2}\right\|^\frac{1}{2}
\left\|\int_{\Omega}| g(t)|^{2}d\mu(t)-\left|\int_{\Omega}g(t)d\mu(t)\right|^{2}\right\|^\frac{1}{2}.
\end{multline}
\end{corollary}

\section*{Acknowledgment}
The author would like to thank the referee for some useful comments and suggestions.


\end{document}